%% file: ms.tex
\documentclass[a4paper,10pt]{article}
\usepackage{mypackages}
\usepackage{mymacros}
\usepackage{mystyle}
\title{A twisted Version of controlled K-Theory}
\author{Elisa Hartmann}

\begin{document}

\maketitle

\begin{abstract}
There are a number of (co-)homology theories on coarse spaces. Controlled operator $K$-theory is by far the most popular one of them. Our approach is geometric. We study when does the Roe-algebra of a space restrict to a subspace. Then we show the Roe-algebra is a cosheaf on the coarse topology. A result is a Mayer-Vietoris exact sequence in the presence of a coarse
cover. We compute examples as an application.

AMS classification: 19K56, 51F99 
\end{abstract}

\tableofcontents

\input{Witpa}
\input{Cosheaves}
\input{Modified_Roe-Algebra}
\input{Computing_Examples}

\bibliographystyle{halpha-abbrv}
\bibliography{mybib}

\address

\end{document}

%% file: Witpa.tex
\section{Introduction}

Controlled operator $K-$theory is one of the most popular homological invariants on coarse metric spaces. Meanwhile a new cohomological invariant on coarse spaces recently appeared in \cite{Hartmann2017a} which studies sheaf cohomology on coarse spaces.

%Background an related Theories

In this paper we study the controlled $K$-theory of a proper metric space $X$ which is introduced in~\cite[Chapter~6.3]{Higson2000}. Note that this theory does not appear as a derived functor as far as we know.

In \cite{Higson1993} is studied a coarse excisive property on coarse spaces which we recall now. If $Y\s X$ is a closed subspace then $C^*(Y,X)$ denotes the ideal in $C^*(X)$ which is the norm closure of operators with support near $Y$. Let $A,B\s X$ be two closed subsets of a proper metric space which are $\omega$-excisive. Then 
\[
\xymatrix{
 K_1(C^*(A\cap B,X))\ar[r]
 & K_1(C^*(A,X))\oplus K_1(C^*(B,X))\ar[r]
 & K_1(C^*(X))\ar[d]\\
 K_0(C^*(X))\ar[u]
 & K_0(C^*(A,X))\oplus K_0(C^*(B,X))\ar[l] 
 & K_0(C^*(A\cap B,X))\ar[l]
 }
\]
is a six-term Mayer-Vietoris exact sequence by \cite[Section~5]{Higson1993}.

Note that our approach can be compared with \cite{Roe2013} where it was shown that a quotient $D(X)/C^*(X)$ is a sheaf on the underlying topological space of $X$.

%Main Contributions

The paper \cite{Hartmann2017a} introduced a Grothendieck topology $X_{ct}$ associated to a coarse space $X$. The underlying category of $X_{ct}$ is the poset of subsets of $X$ and the coverings are finite collections of subsets called \emph{coarse covers}.

The Theorem~\ref{thm:roealgebracosheaf} shows if $X$ is a proper metric space then the association
\[
U\mapsto C^*(\bar U)/\compactops{\hilbert_U}
\]
for every subset $U\s X$ with restriction maps is a cosheaf on $X_{ct}$. 

Note that in a general setting cosheaves with values in the category of abelian groups $\abel$ do not give rise to a derived functor. In~\cite{Curry2014} is explained that the dual version of sheafification, cosheafification, does not work in general. Moreover the category of $C^*$-algebra $\cstar$ is not abelian.

Our result gives rise to new computational tools one of which is a new Mayer-Vietoris six-term exact sequence which is Corollary~\ref{cor:ktheorymv}: If $U_1,U_2\s X$ are subsets of a proper metric space that coarsely cover a subspace $U\s X$ then
\[
\xymatrix{
 K_1(\hat C^*(U_1\cap U_2))\ar[r]
 & K_1(\hat C^*(U_1))\oplus K_1(\hat C^*(U_2))\ar[r]
 & K_1(\hat C^*(U))\ar[d]\\
 K_0(\hat C^*(U))\ar[u]
 & K_0(\hat C^*(U_1))\oplus K_0(\hat C^*(U_2))\ar[l] 
 & K_0(\hat C^*(U_1\cap U_2))\ar[l]
 }
\]
is exact. Here $\hat C^*(Y)=C^*(\bar Y)/\compactops{\hilbert_Y}$ for $Y\s X$ a subset.

The outline of this paper is as follows: The Chapter~\ref{sec:cosheaf} discusses cosheaves on coarse spaces. The main part of the study is in Chapter~\ref{sec:ckt} and Chapter~\ref{sec:exs} computes examples.

%% file: Cosheaves.tex
\section{Cosheaves}

We recall~\cite[Definition~45]{Hartmann2017a}:
\begin{defn}\name{coarse cover}
 If $X$ is a metric space and $U\s X$ a subset a finite family of subsets $U_1,\ldots,U_n\s U$ is said to 
\emph{coarsely cover} $U$ if for every entourage $E\s X^2$ there is a bounded set $B\s X$ such that
\[
 U^2\cap (\bigcup_i U_i^2)^c\cap E\s B^2.
\]
Coarse covers determine a Grothendieck topology $X_{ct}$ assoociated to a metric space $X$. If $f:X\to Y$ is a coarse map between metric spaces then there is a morphism of Grothendieck topologies $\ii f:Y_{ct}\to X_{ct}$.
\end{defn}

\label{sec:cosheaf}
\begin{defn}\name{precosheaf}
 A \emph{precosheaf on $\coarsetopology X$ with values in a category $C$} is a covariant functor 
$Cat(\coarsetopology X)\to C$.
\end{defn}

\begin{defn}\name{cosheaf}
 Let $C$ be a category with finite limits and colimits. A precosheaf $\sheaff$ on $\coarsetopology X$ with 
values 
in $C$ is 
a \emph{cosheaf on $\coarsetopology X$ with values in $C$} if for every coarse cover $\{U_i\to U\}_i$ there is 
a 
coequalizer diagram:
 \[
  \opij\sheaff U\rightrightarrows \opi \sheaff U \to \sheaff(U)  
 \]
Here the two arrows on the left side relate to the following 2 diagrams:
\[
 \xymatrix{
 \bigoplus_{i,j}\sheaffp {U_i\cap U_j}\ar[r]
 & \bigoplus_i \sheaffp {U_i}\\
  \sheaffp {U_i\cap U_j}\ar[u]\ar[r]
 & \sheaffp {U_i}\ar[u]
 }
\]
and
\[
 \xymatrix{
 \bigoplus_{i,j}\sheaffp {U_i\cap U_j}\ar[r]
 & \bigoplus_i \sheaffp {U_j}\\
  \sheaffp {U_i\cap U_j}\ar[r]\ar[u]
 & \sheaffp {U_j}\ar[u]
 }
\]
where $\bigoplus$ denotes the coproduct over the index set.
\end{defn}

\begin{notat}
If we write
 \begin{itemize}
  \item $\sum_i a_i \in \opi \sheaff U$ then $a_i$ is supposed to be in $\sheaffp{U_i}$
  \item $\sum_{ij} b_{ij}\in\opij \sheaff U$ then $b_{ij}$ is supposed to be in 
$\sheaffp{U_i\cup U_j}$
 \end{itemize}
\end{notat}

\begin{prop}
  If $\sheaff$ is a precosheaf on $\coarsetopology X$ with values in a category $C$ with finite limits and 
colimits and for every 
coarse cover $\{U_i\to U\}_i$
 \begin{enumerate}
  \item  and every $a\in \sheaffp U$ there is some $\sum_i a_i\in\opi 
\sheaff U$ such that $\sum_i a_i|_U=a$
\item and for every $\sum_i a_i\in\opi\sheaff U$  such that $\sum_i a_i|_U=0$ there is some $\sum_{ij} 
b_{ij}\in\opij\sheaff U$ such that $(\sum_j b_{ij}-b_{ji})|_{U_i}=a_i$ for every $i$.
 \end{enumerate}
 then $\sheaff$ is a cosheaf.
\end{prop}
\begin{proof}
 easy.
\end{proof}

\begin{rem}
  Denote by $\cstar$ the category of $C^*$-algebras. According to~\cite{Pedersen1999} all finite limits and finite 
colimits exist in $\cstar$.
\end{rem}

%% file: Modified_Roe-Algebra.tex
\section{Modified Roe-Algebra}
\label{sec:ckt}
\begin{lem}
\label{lem:reps}
 If $X$ is a proper metric space and $Y\s X$ is a closed subspace then 
 \begin{itemize}
 \item the subset $I(Y)=\{f\in C_0(X):f|_Y=0\}$ is an ideal of $C_0(X)$ and we have
 \[
  C_0(Y)=C_0(X)/I(Y)
 \]
 \item we can restrict the non-degenerate representation $\rho_X:C_0(X)\to \boundedops {\hilbert_X}$ to a representation
 \[
  \rho_Y:C_0(Y)\to \boundedops {\hilbert_Y}
 \]
 in a natural way.
 \item the inclusion $i_Y:\hilbert_Y\to \hilbert_X$ covers the inclusion $i:Y\to X$.
 \end{itemize}
\end{lem}
\begin{proof}
 \begin{itemize}
  \item This one follows by Gelfand duality.
  \item We define $\hilbert_{I(Y)}=\overline{\rho_X(I(Y))\hilbert_X}$. Then
  \[
   \hilbert_X=\hilbert_{I(Y)}\oplus \hilbert_{I(Y)}^\perp
  \]
is the direct sum of reducing subspaces for $\rho_X(C_0(X))$. We define
\[
 \hilbert_Y=\hilbert_{I(Y)}^\perp
\]
and a representation of $C_0(Y)$ on $\hilbert_Y$ by
\[
 \rho_Y([a])=\rho_X(a)|_{\hilbert_Y}
\]
for every $[a]\in C_0(Y)$. Note that $\rho_X(\cdot)|_{\hilbert_Y}$ annihilates $I(Y)$ so this is well defined.
\item Note that the support of $i_Y$ is 
\f{
 \supp(i_Y)&=\Delta_Y\\
 &\s X\times Y 
}
\end{itemize}
\end{proof}

\begin{rem}
Note that we can not conclude the following: If the representation $\rho_X:C_0(X)\to \boundedops{\hilbert_X}$ is ample and $Y\s X$ is a closed subspace then the induced representation $\rho_Y:C_0(Y)\to \boundedops{\hilbert_Y}$ is ample. There are counterexamples to this claim.
\end{rem}

\begin{lem}
 If $X$ is a proper metric space,
 \begin{itemize}
 \item $B\s X$ a compact subset and $T\in C^*(X)$ an operator with
 \[
  \supp T\s B^2
 \]
 then $T$ is a compact operator.
 \item The converse does not hold. If $T\in C^*(X)$ is a compact operator then there does not necessarily exist a bounded set $B\s X^2$ such that $\supp T\s B^2$
 \item The $C^*$-algebra of compact operators $\compactops{\hilbert_X}$ is an ideal in $C^*(X)$.
 \end{itemize}
\end{lem}
\begin{proof}
\begin{itemize}
 \item Suppose there is a non-degenerate representation $\rho:C_0(X)\to \boundedops {\hilbert_X}$. For every $f\in C_0(B^c),g\in C_0(X)$ the equations $\rho(f)T\rho(g)=0$ and $\rho(g)T\rho(f)=0$ hold. This implies $T(I(B))=0$ and $\im T\cap I(B)=0$. Thus $T:\hilbert_B\to \hilbert_B$ is the same map. Thus $T\in C^*(B)$ already. Now $T$ is locally compact, $B$ is compact thus $T$ is a compact operator.
 \item Note the set of ghost operator as defined in \cite[Definition~1.2]{Roe2014} contains the compact operators. The space $X$ has property A if and only if every ghost operators is compact by \cite[Theorem~1.3]{Roe2014}. Not all of them have bounded support. The paper \cite{Drutu2019} shows there exists a metric space $\mathcal O$ with non-compact ghost projections in $\boundedops{L_2(\mathcal O)}$. This implies that $\mathcal O$ does not coarsely embed into Hilbert space.
 \item This is already \cite[Lemma~4.12]{Roe1993}. For the convenience of the reader we recall the proof: Let $K\s X$ be a set and $v\in\hilbert_X$ be a vector with $\supp v\s K$. Then for every $f\in I(K)$ we obtain $\rho(f)v=0$. Now a vector in $\hilbert_{I(Y)}$ can be written as $\rho(f)w$ where $f\in I(Y),w\in\hilbert_X$. Then
 \f{
 \langle v,\rho(f)w\rangle
 &=\langle \rho(f^*)v,w\rangle\\
 &=0
 }
 Thus $v\in\hilbert_K$. If on the other hand $v'\in\hilbert_X$ is any vector then
 \f{
 0
 &=\langle \rho(f)w,v'\rangle\\
 &=\langle w,\rho(f^*)v'\rangle
 }
 for every $f\in I(K),w\in \hilbert_X$. This implies $\rho(f^*)v'=0$ for every $f\in I(K)$. Thus $\supp v'\s K$. Since $X$ can be written as a union of bounded sets, $X=\bigcup B_i$ with $B_i$ bounded for every $i$, the vectors with compact support form an orthonormal basis of $\hilbert_X$.
 
A finite rank operator $T$ with respect to this basis belongs to $C^*(X)$: First of all $T$ is locally compact since it is compact. We can write
\[
T:h\mapsto\sum_{i=1}^n\alpha_i\langle h,v_i\rangle u_i
\]
here $\alpha_i\ge 0$ and $v_i,u_i$ are vectors with compact support $\supp v_i\s B_i,\supp u_i\s A_i$ for $1\le i\le n$. Let $v\in\hilbert_X$ be a vector. If $g\in I(B_i)$ and $f\in C_0(X)$ are two functions then
\f{
\alpha_i\langle \rho(g)v,v_i\rangle \rho(f) u_i
&=\alpha_i\langle v,\rho(g^*)v_i\rangle\rho(f)u_i\\
&=0
}
Now let $g\in C_0(X),f\in I(A_i)$ be functions. Then $\alpha_i\langle \rho(g)v,v_i\rangle\rho(f) u_i=0$ since $\supp u_i\s A_i$. Thus $\supp T\s \bigcup_{i=1}^n A_i\times B_i$ is controlled.

Since the finite rank operators are dense in the compact operators $\compactops {\hilbert_X}$ we obtain the inclusion $\compactops{\hilbert_X}\s C^*(X)$. Since the composition with a compact operator yields a compact operator the subset $\compactops{\hilbert_X}$ is an ideal in $C^*(X)$.
 \end{itemize}
\end{proof}

\begin{defn}\name{modified Roe-algebra}
Let $X$ be a proper metric space then
 \[
  \hat C^*(X)=C^*(X)/\compactops {\hilbert_X}
  \]
 where $\compactops {\hilbert_X}$ denotes the compact operators of $\boundedops{\hilbert_X}$ is called the 
\emph{modified Roe-algebra} of $X$. 
\end{defn}

\begin{rem}
 If $U\s X$ is a subset of a proper metric space then $U$ is coarsely dense in $\bar U$. We define
 \[
  \hat C^*(U):=\hat C^*(\bar U)
 \]
 Note that makes sense because if $U_1,U_2$ is a coarse cover of $U$ then $\bar U_1,\bar U_2$ is a coarse cover of 
$\bar U$ also.
 \end{rem}

 \begin{lem}
  If $Y\s X$ is a closed subspace and $i_Y:\hilbert_Y\to \hilbert_X$ the inclusion operator of Lemma~\ref{lem:reps} then
  \begin{itemize}
 \item the operator
\f{
 Ad(i_Y):C^*(Y)&\to C^*(X)\\
 T&\mapsto i_Y T i_Y^*
}
is well defined and maps compact operators to compact operators.
\item Then the induced operator on quotients
\[
 \hat {Ad}(i_Y):\hat C^*(Y)\to \hat C^*(X)
\]
is the dual version of a restriction map.
\end{itemize}
\end{lem}
\begin{proof}
 \begin{itemize}
  \item $i_Y$ covers the inclusion the other statement is obvious.
  \item easy.
 \end{itemize}
\end{proof}

\begin{thm}
\label{thm:roealgebracosheaf}
 If $X$ is a proper metric space then the assignment
 \[
  U \mapsto \hat C^*(U)
  \]
  for every subspace $U\s X$ is a cosheaf with values in $\cstar$.
\end{thm}
\begin{proof}
Let $U_1,U_2\s U$ be subsets that coarsely cover $U\s X$ and $V_1:\hilbert_{U_1}\to \hilbert_U$ and $V_2:\hilbert_{U_2}\to \hilbert_U$ the corresponding inclusion operators.
 \begin{enumerate}
  \item Let $T\in C^*(U)$ be a locally compact controlled operator. We need to construct $T_1\in C^*(U_1),T_2\in C^*(U_2)$ such that
\[
 V_1T_1V_1^*+V_2T_2V_2^*=T
\]
modulo compacts. Denote by 
\[
E=\supp(T)
\]
the support of $T$ in $U$. Define
\[
 T_1:=V_1^*TV_1
\]
and
\[
 T_2:=V_2^*TV_2
\]
then it is easy to check that $T_1,T_2$ are locally compact and controlled operators, thus elements in $C^*(U_1),C^*(U_2)$. Now $\supp(V_1T_1V_1^*)= U_1^2\cap E$ and $\supp(V_2T_2V_2^*)= U_2^2\cap E$. Thus
\f{
 \supp( V_1T_1V_1^*+V_2T_2V_2^*-T)
 &=E\cap (U_1^2\cup U_2^2)^c\\
 &\s B^2
}
where $B$ is bounded. This implies $T_1|_U+T_2|_U=T$.
 \item Suppose there are $T_1\in C^*(U_1), T_2\in C^*(U_2)$ such that
 \[
  V_1T_1V_1^*+V_2T_2V_2^*=0
 \]
 modulo compacts. That implies that $\supp(V_1T_1V_1^*)\s (U_1\cap U_2)^2$ modulo bounded sets and $\supp(V_2T_2V_2^*)\s(U_1\cap U_2)^2$ modulo bounded sets. Also $V_1T_1V_1^*=-V_2T_2V_2^*$ modulo compacts. Denote by $V^i_{12}:\hilbert_{U_1\cap U_2}\to \hilbert_{U_i}$ the inclusion for $i=1,2$. Define
\[
 T_{12}=V_{12}^*T_1V_{12}
\]
 Then 
\f{
 V_{12}^1T_{12}V_{12}^{1*}
 &=V_{12}^1V_{12}^{1*}T_1V_{12}^1V_{12}^{1*}\\
 &=T_1
}
Then $V_1\circ V^1_{12}=V_2\circ V^2_{12}$ implies
\f{
 V_{12}^2T_{12}V_{12}^{2*}
 &=V_2^*V_2V^2_{12}T_{12}V_{12}^{2*}V_2^*V_2\\
 &=V_2^*V_1V^1_{12}T_{12}V_{12}^{1*}V_1^*V_2\\
 &=V_2^*V_1T_1V_1^*V_2\\
 &=-T_2
}
modulo compacts.
 \end{enumerate}
\end{proof}

%% file: Computing_Examples.tex
\section{Computing Examples}
\label{sec:exs}
\begin{cor}
\label{cor:ktheorymv}
 If $U_1,U_2$ coarsely cover a subset $U$ of a proper metric space $X$ then there is a six-term Mayer-Vietoris exact 
sequence
\[
 \xymatrix{
 K_1(\hat C^*(U_1\cap U_2))\ar[r]
 & K_1(\hat C^*(U_1))\oplus K_1(\hat C^*(U_2))\ar[r]
 & K_1(\hat C^*(U))\ar[d]\\
 K_0(\hat C^*(U))\ar[u]
 & K_0(\hat C^*(U_1))\oplus K_0(\hat C^*(U_2))\ar[l] 
 & K_0(\hat C^*(U_1\cap U_2))\ar[l]
 }
\]
\end{cor}
\begin{proof}
 We show there is a pullback diagram of $C^*-$algebras and $*$-homomorphisms
 \[
 \xymatrix{
  \hat C^*(U_1\cap U_2)\ar[d]\ar[r]
  & \hat C^*(U_2)\ar[d]\\
  \hat C^*(U_1)\ar[r]
  &\hat C^*(U)
  }
 \]
The result is then an application of~\cite[Exercise~4.10.22]{Higson2000}.

Let $A$ be a $C^*$-algebra with $*$-homomorphisms $f:A\to \hat C^*(U_1),g:A\to \hat C^*(U_2)$ such that 
\[
i_{U_1}f(a)i_{U_1}^*=i_{U_2}g(a)i_{U_2}^*
\]
modulo compacts for every $a\in A$. By Theorem~\ref{thm:roealgebracosheaf} we can use the second cosheaf-axiom on $f(a)-g(a)\in \hat C^*(U_1)\oplus \hat C^*(U_2)$ which restricts to $0$ on $U$. Thus there is some $\sum_{ij} b_{ij}\in \bigoplus_{ij}\hat C^*(U_i\cap U_j)$ such that $(\sum_j b_{1j}-b_{j1})|_{U_1}=f(a)$ and $(\sum_j b_{2j}-b_{j2})|_{U_2}=-g(a)$. Then define a map $h:A\to \hat C^*(U_1\cap U_2)$ by
\[
h(a)=b_{12}-b_{21}
\]
for every $a\in A$. Since the restriction maps associated to the inclusions $U_1\cap U_2\to U_i$ are injective for $i=1,2$ the mapping $h$ is unique and a $*$-homomorphism.
\end{proof}

\begin{rem}
\label{rem:modktheory}
 Now for every proper metric space there is a short exact sequence
 \[
  0\to\compactops{\hilbert_X}\to C^*(X)\to \hat C^*(X)\to 0
 \]
which induces a 6-term sequence in K-theoy:
\[
 \xymatrix{
 K_0(\compactops{\hilbert_X})\ar[r]
 & K_0(C^*(X))\ar[r]
 & K_0(\hat C^*(X))\ar[d]\\
 K_1(\hat C^*(X))\ar[u]
 & K_1(C^*(X))\ar[l] 
 & K_1(\compactops{\hilbert_X})\ar[l]
 }
\]
If $X$ is flasque then
\[
 K_i(\hat C^*(X))=\begin{cases}
                   0 & i=0\\
                   \Z & i=1
                  \end{cases}
\]
\end{rem}

\begin{rem}
Note that the result of Corollary~\ref{cor:ktheorymv} is applicable when computing controlled $K$-theory if the property ample is preserved by restricting the representation of $U$ to the representations of $U_1,U_2$.

If $X$ is a Riemannian manifold then $L^2(X)$ is a Hilbert space. In Example~\ref{ex:ex1}, Example~\ref{ex:ex2} we will use the canonical representations of type $C_0(X)\to \boundedops{L^2(X)}$ on $\R,\R^2$ and certain subspaces of them without mentioning it. In those cases the property ample is preserved by restricting $\R$ to $\R_+$, the space $\R^2$ to $V_1,V_2$ and $V_1$ to $U_1,U_2$ respectively.
\end{rem}

\begin{ex}\name{$\R$}
\label{ex:ex1}
 Now $\R$ is the coarse disjoint union of two copies of $\R_+$ which is a flasque space. By Corollary 
\ref{cor:ktheorymv} there is an isomorphism
\[
 K_i(\hat C^*(\R))=\begin{cases}
                    0 & i=0\\
                    \Z\oplus \Z & i=1
                   \end{cases}
\]
Then it is a result of Remark~\ref{rem:modktheory} that there is an isomorphism
\[
 K_i(C^*(\R))=\begin{cases}
               0 & i=0\\
               \Z & i=1
              \end{cases}.
\]
This result is not surprising, it matches the computation in \cite[Theorem~6.4.10]{Higson2000}.
\end{ex}

\begin{ex}\name{$\R^2$}
\label{ex:ex2}
We coarsely cover $\R^2$ with 
 \[
  V_1=\R_+\times \R \cup\R\times \R_+
 \]
and
\[
 V_2=\R_-\times \R\cup \R\times \R_-.
\]
then again $V_1$ is coarsely covered by
\[
 U_1=\R_+\times \R
\]
and
\[
 U_2=\R\times \R_+
\]
and $V_2$ is coarsely covered in a similar fashion. We first compute modified controlled K-theory of $V_1$ and then of 
$\R^2$. Note that the inclusion $U_1\cap U_2\to U_1$ is split by
\f{
r:\R_+\times\R &\to \R_+^2\\
(x,y)&\mapsto (x,|y|)
}
Thus using Corollary~\ref{cor:ktheorymv} we conclude that
\[
 K_i(\hat C^*(V_j))=\begin{cases}
           0 & i=0\\
           \R & i=1
          \end{cases}
\]
for $j=1,2$. Then again using Corollary~\ref{cor:ktheorymv} and that the inclusion $\R_+^2\to V_i$ is split we can 
compute
\[
 K_i(\hat C^*(\R^2))=\begin{cases}
                      0 &i=0\\
                      0 &i=1
                     \end{cases}
\]
Translating back we obtain
\[
 K_i(C^*(\R^2))=\begin{cases}
                      \Z &i=0\\
                      0 &i=1
                     \end{cases}
\]
This one also fits the computations in \cite[Theorem~6.4.10]{Higson2000}.
\end{ex}